\documentclass{article}%
\usepackage{amssymb}
\usepackage{amsfonts}
\usepackage{amsmath}
\usepackage{graphicx}
\usepackage{amssymb,color}%
\providecommand{\U}[1]{\protect\rule{.1in}{.1in}}

\definecolor{c20}{rgb}{0.,0.0,0.}
\definecolor{c30}{rgb}{0.,0.,0.}
\definecolor{c40}{rgb}{0,0.0,0.0}
\definecolor{c50}{rgb}{0,0,0}

\newtheorem{theorem}{Theorem}

\newtheorem{conclusion}{Conclusion}

\newtheorem{corollary}{Corollary}

\newtheorem{example}{Example}

\newtheorem{lemma}[theorem]{Lemma}

\newtheorem{proposition}{Proposition}
\newtheorem{remark}{Remark}

\newenvironment{proof}[1][Proof]{\noindent\textbf{#1.} }{\ \rule{0.5em}{0.5em}}
\begin{document}

\title{Joint distribution of a random sample and an order statistic: A new
approach with an application in reliability analysis}
\author{ Ismihan Bairamov\thanks{%
Department of Mathematics, Faculty of Arts and Sciences, Izmir University of
Economics, 35330, Balcova, Izmir, Turkey, \textit{%
ismihan.bayramoglu@ieu.edu.tr}}, {\small \textit{Izmir University of
Economics, Turkey}}}
\maketitle

\begin{abstract}
This paper considers the joint distribution of elements of a random sample
and an order statistic of the same sample. \ The motivation for this work
stems from the important problem in reliability analysis, to estimate the
number of inspections we need in order to detect failed components in a
coherent system. We consider an $(n-r+1)$-out-of-$n$ system, which is intact
until at least $n-r+1$ of the components are alive, and it fails if the
number of failed components exceeds $r$. The life time of the system is the $%
r$th order statistic. Assuming that some of the components failed but the
system \ is still functioning, \ using the results presented in this paper
it is possible to find an expected value of the number of inspections we
need to do for detecting certain number of failed components.

\end{abstract}

\section{Introduction}

Let $X_{1},X_{2},...,X_{n}$ be independent and identically distributed (iid)
random variables with distribution function (cdf) $F$ and $\ X_{1:n}\leq
X_{2:n}\leq \cdots \leq X_{n:n}$ be the order statistics. If $%
X_{1},X_{2},...,X_{n}$ are corresponding lifetimes of components of a
coherent system, then for $1\leq r\leq n$\ the conditional probability
\begin{equation}
P\{X_{1}\leq x\mid X_{r:n}\leq t\}  \label{1}
\end{equation}%
is the distribution of lifetime of any of components given that at the
inspection time $t$ at least $r$ of the components have failed. The
conditional distribution
\begin{equation}
P\{X_{1}\leq x\mid X_{r:n}=t\}  \label{2}
\end{equation}%
is studied in Nagaraja and Nevzorov (1997) and Nagaraja and Ahmadi (2018) in
the context of $(n-r+1)-$out-of-$n$ systems whose lifetime $%
T(X_{1},X_{2},...,X_{n})$ is $X_{r:n},$ i.e. the system that fails if more
than $r$ components fail and the system is intact if at least $(n-r+1)$ of
components are alive. The probability (\ref{2}) is actually the conditional
cdf of any of the components of $(n-r+1)$-out-of-$n$ coherent system given
that the system failed at time $t.$ Nagaraja and Ahmadi (2018) \ used (\ref%
{2}) and related joint distributions to find the distribution of number of
inspections which is necessary for detecting of all failed components if the
system failed at time $t.$

In this paper we are interested also in conditional distribution
\begin{equation}
P\{X_{1}\leq x\mid t_{1}\leq X_{r:n}\leq t_{2}\},  \label{3}
\end{equation}%
which can be interpreted as the conditional distribution of any of the
components given that the $r$th \ failure has occurred between two
inspections at$\ t_{1}$ and $t_{2},$ i.e. there are $r$ failed components
that we reveal in time interval $[t_{1},t_{2}].$ \ The conditional
distribution (\ref{3}) \ carries information about the life time
distribution of any of the components given that the $r$th failure has
occurred between two inspection times $t_{1}$and $t_{2}.$ The random
variables $X_{1},X_{2},...,X_{n}$ can also be considered as the lifetimes of
$n$ identical items put under life test and then (\ref{1}) is the
conditional distribution of any of items given that at inspection time $t,$
there are at least $r$ failed items. \ In \ practical applications a system
monitoring is important, and it is scheduled at different inspection times.
Under double monitoring one may \ consider the residual and past life
functions of the system $(T-t_{1}\mid t_{1}<_{1}T<t_{2})$ and ($t_{2}-T\mid
t_{1}<T<t_{2})$ which is studied in many research papers including Raqab
(2010), Bdair and Raqab (2014), Li and Zhao (2008), Li and Zhang (2008),
Parvardeh et al. (2018), Poursaeed (2010), Poursaeed and Nemathollahi
(2010a), Poursaeed and \ Nemathollahi (2010b), Zhang and Meeker (2013), and
Zhang and Yang (2010), Ery\i lmaz (2013), Tavangar and Bairamov (2015),
Samadi et al. (2017). \ In a recent paper Navarro \ and Cali (2018) \
consider a system with dependent components assuming that system is exposed
to periodical inspections. In the results of these inspections, it may be
known that the system was working at time $t_{1},$ but it failed at time $%
t_{2}.$ Under these conditions Navarro \ and Cali (2018) investigate the
system inactivity time ($t_{2}-T\mid t_{1}<T<t_{2})$ for both independent
and dependent lifetimes and obtain representations for reliability functions
in terms of copula.

The focus of this paper is the\ joint distribution of $X_{1},X_{2},...,X_{k}$
and $X_{r:n}$ for $k<r.$ $\ $First, we consider the joint distribution of $%
X_{i}$ and $X_{r:n},$ $i\in \{1,2,...,n\}$ $\ $as well as the conditional
distribution of $X_{i}$ given $X_{r:n}\leq t$ and derive the conditional
distribution of $X_{1}$ given $t_{1}\leq X_{\dot{r}:n}\leq t_{2}$ for any $%
t_{1}<t_{2}.$ \ The difficulty of finding the joint distribution of random
variables $X_{1},X_{2},...,X_{k}$ and $X_{r:n}$ \ is concluded in the fact
that $X_{r:n}$ is one of the random variables $X_{1},X_{2},...,X_{n}.$ \
Second, we apply the obtained results to solve an important problem in
reliability analysis: the problem of estimating the number of inspections we
need in order to detect failed components of a coherent system. Since
inspections of the components of the system may sometimes be an expensive
action, the optimal planning of periodical inspections is very important. If
we interpret $X_{i}$'s as the lifetimes of components of $(n-r+1)$-out-of-$n$
system, then the joint distribution of the random variables $%
X_{1},X_{2},...,X_{k}$ and $X_{r:n}$ is necessary to compute the
probabilities of the events of type $%
X_{1}<X_{r:n},X_{2}<X_{r:n},...,X_{k}<X_{r:n}$ which are used for computing
the probabilities of numbers of inspections we need in order to detect
failed components. This paper is organized as follows: in Section 1 we
derive the joint distribution of a single observation from the sample and
the $r$th order statistic of the same sample and consider the conditional
distribution of an observation given that the $r$th order statistic is
between $t_{1}$ and $t_{2}.$ Then we consider the joint distributions of
several sample observations and an order statistic of the same sample and
show that the conditional random variables defined as a set of observations
given order statistic are in general dependent, except some special cases. \
In Section 3 we \ consider the joint distributions of the set of sample
observations and an order statistic. In Section 4 we deal with the
distribution of the number of inspections one needs in order to detect
failed components in an $(n-r+1)$-out-of-$n$ system and provide a numerical
example.

\section{The joint distributions of the random variables and their order
statistics}

Throughout this paper we assume that $X_{1},X_{2},...,X_{n}$ be iid random
variables with cdf $F$ $\ $and $X_{1:n}\leq X_{2:n}\leq \cdots \leq X_{n:n}$
be the order statistics. Where it is needed we will assume that $F$ is an
absolutely continuous cdf with pdf $F^{\prime }(x)=f(x)$ supported in $%
[0,\infty )$ and $X_{i},i\in \{1,2,...,n\}$ are lifetimes of the components
of coherent system of $n$ components.

\begin{theorem}
\label{Theorem 1}The joint distribution of $X_{1}$ and $X_{r:n}$ is
\begin{eqnarray}
P\{X_{1} &\leq &x,X_{r:n}\leq t\}  \notag \\
&=&\left\{
\begin{array}{c}
\begin{tabular}{ll}
$F(x)\sum\limits_{i=r-1}^{n-1}\binom{n-1}{i}F^{i}(t)(1-F(t))^{n-1-i},$ & $%
x\leq t$ \\
$%
\begin{tabular}{l}
$\left[ F(x)\sum\limits_{i=r-1}^{n-1}\binom{n-1}{i}F^{i}(t)(1-F(t))^{n-1-i}%
\right. $ \\
$\left. -\binom{n-1}{r-1}(F(x)-F(t))F^{r-1}(t)(1-F(t))^{n-r}\right] $%
\end{tabular}%
,$ & $x>t$%
\end{tabular}
\\
\end{array}%
\right.  \label{0}
\end{eqnarray}
\end{theorem}

\begin{proof}
a) Let $x\leq t.$ \ We have
\begin{eqnarray*}
P\{X_{1} &\leq &x,X_{r:n}\leq t\}=P\{X_{1}\leq x,\text{at least \ }r-1\text{
of } \\
&&X_{2},X_{3},...,X_{n}\text{ are less or equal than }t\} \\
&=&F(x)P\{\text{exactly }r-1\text{ of }X_{2},X_{3},...,X_{n}\text{ are less
or equal than }t\} \\
&=&F(x)\sum\limits_{i=r-1}^{n-1}\binom{n-1}{i}F^{i}(t)(1-F(t))^{n-1-i}\text{
}
\end{eqnarray*}

b) Let $x>t.$ Using the total probability formula one can write
\begin{eqnarray*}
P\{X_{1} &\leq &x,X_{r:n}\leq t\} \\
&=&P\{X_{1}\leq x,X_{1}>t,X_{r:n}\leq t\}+P\{X_{1}\leq x,X_{1}\leq
t,X_{r:n}\leq t\} \\
&=&P\{X_{1}\leq x,X_{1}>t,\text{at least \ }r\text{ of }X_{2},X_{3},...,X_{n}%
\text{ are less than or equal to }t\} \\
+P\{X_{1} &\leq &x,X_{1}\leq t,\text{at least \ }r-1\text{ of }%
X_{2},X_{3},...,X_{n}\text{ are less than or equal to }t\}
\end{eqnarray*}%
\begin{eqnarray*}
&=&(F(x)-F(t))\sum\limits_{i=r}^{n-1}\binom{n-1}{i}F^{i}(t)(1-F(t))^{n-1-i}
\\
&&+F(t)\sum\limits_{i=r-1}^{n-1}\binom{n-1}{i}F^{i}(t)(1-F(t))^{n-1-i} \\
&=&(F(x)-F(t))\sum\limits_{i=r-1}^{n-1}\binom{n-1}{i}F^{i}(t)(1-F(t))^{n-1-i}
\\
&&-(F(x)-F(t))\binom{n-1}{r-1}F^{r-1}(t)(1-F(t))^{n-r} \\
&&+F(t)\sum\limits_{i=r-1}^{n-1}\binom{n-1}{i}F^{i}(t)(1-F(t))^{n-1-i}
\end{eqnarray*}%
\begin{eqnarray*}
&=&\sum\limits_{i=r-1}^{n-1}\binom{n-1}{i}%
F^{i}(t)(1-F(t))^{n-1-i}(F(x)-F(t)-F(t)) \\
&&-(F(x)-F(t))\binom{n-1}{r-1}F^{r-1}(t)(1-F(t))^{n-r}.
\end{eqnarray*}%
The theorem is thus proved.
\end{proof}

\begin{corollary}
\label{Corollary 1}The conditional distribution of $X_{1}$ given $%
X_{r:n}\leq t$ is
\begin{eqnarray}
P\{X_{1} &\leq &x\mid X_{r:n}\leq t\}  \notag \\
&=&\left\{
\begin{tabular}{lll}
$F(x)\sum\limits_{i=r-1}^{n-1}\binom{n-1}{i}F^{i}(t)(1-F(t))^{n-1-i}$ &  &
\\
$\times \left( \sum\limits_{i=r}^{n}\binom{n}{i}F^{i}(t)(1-F(t))^{n-i}%
\right) ^{-1}$ & if & $x\leq t$ \\
\begin{tabular}{l}
$\left[ F(x)\sum\limits_{i=r-1}^{n-1}\binom{n-1}{i}F^{i}(t)(1-F(t))^{n-1-i}%
\right. $ \\
$\left. -\binom{n-1}{r-1}(F(x)-F(t))F^{r-1}(t)(1-F(t))^{n-r}\right] $ \\
$\times \left( \sum\limits_{i=r}^{n}\binom{n}{i}F^{i}(t)(1-F(t))^{n-i}%
\right) ^{-1}$%
\end{tabular}
& if & $x>t$%
\end{tabular}%
\right.  \label{4}
\end{eqnarray}
\end{corollary}

\begin{proof}
\bigskip Follows from Theorem 1.
\end{proof}

Below in Figure 1 we provide for illustration the graph of the joint
distribution $P\{X_{1}\leq x,X_{r:n}\leq t\}$ for $F(x)=1-\exp (-x),x\geq
0,n=15$ and $r=7.$
\\\\\

\includegraphics[scale=0.22]{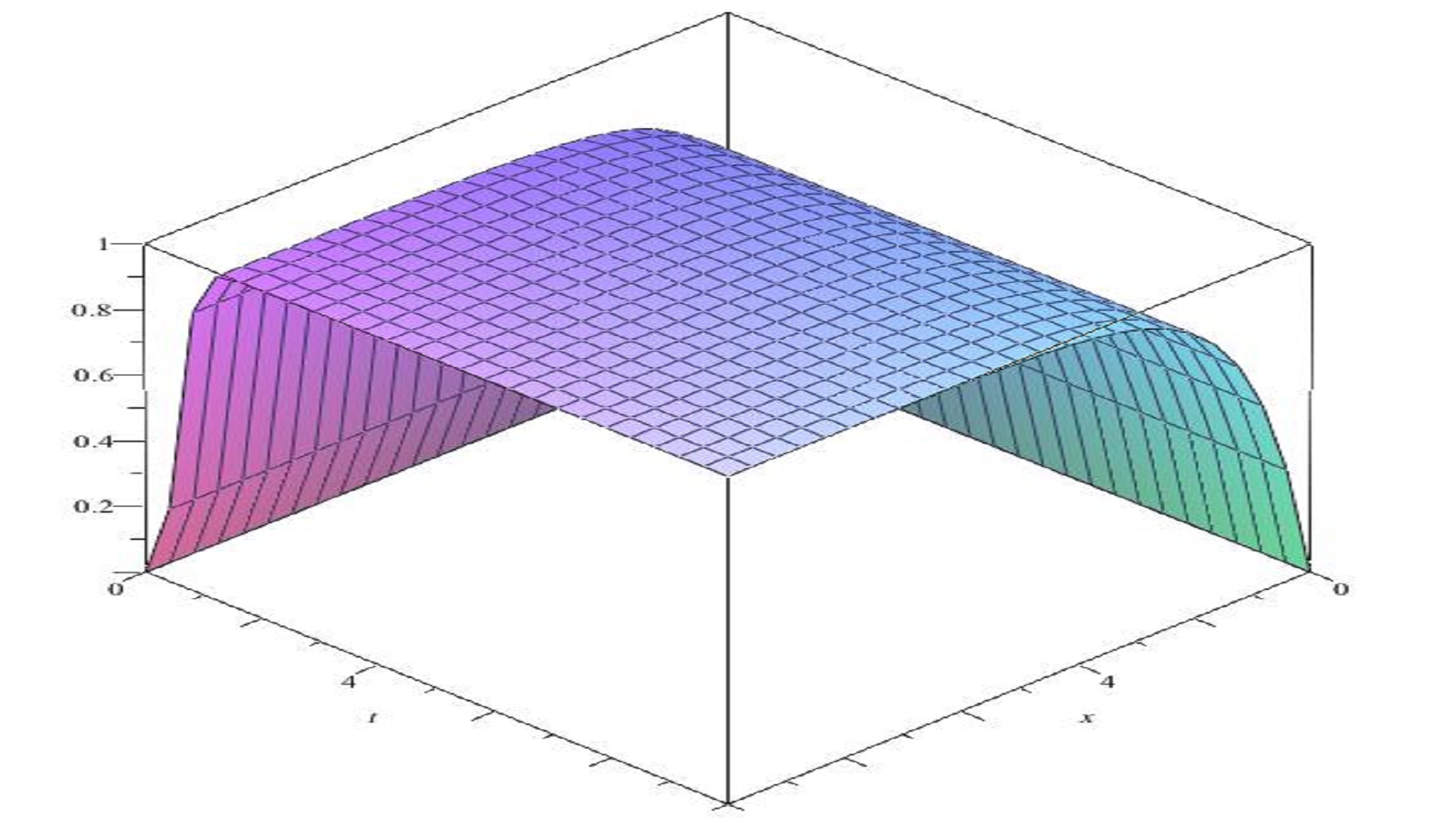}

Figure 1.
The graph of $P\{X_{1}\leq x,X_{r:n}\leq t,$ $n=15,r=7,F(x)=1-\exp (-x),x\geq 0.$

\begin{remark}
\label{Remark 1}Special cases. Because of the importance of formula (\ref{4}%
) (or (\ref{0})) for our research, we can verify it with the special cases $%
r=n$ and $r=1,$ which can be computed by using the properties of extreme
order statistics.
\end{remark}

a) Let $r=n.$ Then if $x<t,$ we have
\begin{eqnarray*}
P\{X_{1} &\leq &x,X_{n:n}\leq t\} \\
&=&P\{X_{1}\leq x,X_{n:n}\leq t\}=P\{X_{1}\leq x,X_{1}\leq t,...,X_{n}\leq
t\} \\
&=&F(x)F^{n-1}(t)
\end{eqnarray*}%
and if $x\geq t$ we have $\ $%
\begin{eqnarray*}
P\{X_{1} &\leq &x,X_{n:n}\leq t\} \\
&=&P\{X_{1}\leq x,X_{n:n}<t\}=P\{X_{1}\leq x,X_{1}\leq t,...,X_{n}\leq t\} \\
&=&F^{n}(t).
\end{eqnarray*}%
Therefore,%
\begin{equation}
P\{X_{1}\leq x,X_{n:n}\leq t\}=\left\{
\begin{tabular}{lll}
$F(x)F^{n-1}(t)$ & if & $x\leq t$ \\
$F^{n}(t).$ & if & $x>t$%
\end{tabular}%
\right.  \label{5}
\end{equation}%
It is clear that
\begin{equation}
P\{X_{1}\leq x\mid X_{n:n}\leq t\}=\left\{
\begin{tabular}{lll}
$\frac{F(x)}{F(t)}$ & if & $x\leq t$ \\
$1$ & if & $x>t$%
\end{tabular}%
\right.  \label{5.0}
\end{equation}

\bigskip Now, let $r=n$ in (\ref{0}) or (\ref{4}) and we clearly obtain (\ref%
{5.0}) or (\ref{5}).

b) Let $r=1.$ Then one can write
\begin{eqnarray*}
P\{X_{1} &\leq &x,X_{1:n}\leq t\} \\
&=&P\{X_{1}\leq x,X_{1:n}\leq t\} \\
&=&P\{X_{1}\leq x\}-P\{X_{1}\leq x,X_{1:n}>t\} \\
&=&F(x)-P\{X_{1}\leq x,X_{1}>t,...,X_{n}>t\}
\end{eqnarray*}%
\begin{equation}
=\left\{
\begin{tabular}{lll}
$F(x)$ & if & $x\leq t$ \\
$F(x)-(F(x)-F(t))(1-F(t)^{n-1}$ & if & $x>t$%
\end{tabular}%
\right. .  \label{6}
\end{equation}%
It is clear that
\begin{eqnarray}
P\{X_{1} &\leq &x\mid X_{1:n}\leq t\}  \notag \\
&=&\left\{
\begin{tabular}{lll}
$\frac{F(x)}{1-(1-F(t))^{n}}$ & if & $x\leq t$ \\
$\frac{F(x)-(F(x)-F(t))(1-F(t)^{n-1}}{1-(1-F(t))^{n}}$ & if & $x>t$%
\end{tabular}%
\right.  \label{6.0}
\end{eqnarray}%
Now, let $r=1$ in (\ref{0}) and (\ref{4}) and one obtains (\ref{6}) and (\ref%
{6.0}).

\begin{remark}
\label{Remark 2}Let $I_{p}(a,b)=\frac{1}{B(a,b)}\int%
\limits_{0}^{p}t^{a-1}(1-t)^{b-1}$ be an incomplete beta function and $%
B(a,b)=\int\limits_{0}^{1}t^{a-1}(1-t)^{b-1}dt$ \ be a beta function. \
Since, $\sum\limits_{i=r}^{n-1}\binom{n-1}{i}%
F^{i}(t)(1-F(t))^{n-1-i}=I_{F(t)}(r,n-r)$ and $\sum\limits_{i=r-1}^{n-1}%
\binom{n-1}{i}F^{i}(t)(1-F(t))^{n-1-i}=I_{F(t)}(r-1,n-r+1),$ for $1<r<n$ (%
\ref{4}) can also be written as
\begin{eqnarray}
P\{X_{1} &\leq &x\mid X_{r:n}\leq t\}=  \notag \\
&&\left\{
\begin{tabular}{lll}
$%
\begin{tabular}{l}
$F(x)I_{F(t)}(r-1,n-r+1)$ \\
$\times \left( I_{F(t)}(r,n-r+1)\right) ^{-1}$%
\end{tabular}%
$ & if & $x\leq t$ \\
&  &  \\
$%
\begin{tabular}{l}
$\lbrack F(x)I_{F(t)}(r-1,n-r+1)$ \\
$-\binom{n-1}{r-1}(F(x)-F(t))F^{r-1}(t)(1-F(t))^{n-r}]$ \\
$\times \left( I_{F(t)}(r,n-r+1)\right) ^{-1}$%
\end{tabular}%
$ & if & $x>t$%
\end{tabular}%
\right.  \label{4b}
\end{eqnarray}%
\begin{equation}
=\left\{
\begin{tabular}{lll}
$%
\begin{tabular}{l}
$F(x)I_{F(t)}(r-1,n-r+1)$ \\
$\times \left( I_{F(t)}(r,n-r+1)\right) ^{-1}$%
\end{tabular}%
$ & if & $x\leq t$ \\
&  &  \\
$%
\begin{tabular}{l}
$\{[F(x)-F(t)]I_{F(t)}(r,n-r)$ \\
$+F(t)I_{F(t)}(r-1,n-r+1)\}$ \\
$\times \left( I_{F(t)}(r,n-r+1)\right) ^{-1}$%
\end{tabular}%
$ & if & $x>t$%
\end{tabular}%
\right. .  \label{4a}
\end{equation}%
where $I_{p}(a,b)=\frac{1}{B(a,b)}\int\limits_{0}^{p}t^{a-1}(1-t)^{b-1}$ is
an incomplete beta function and $B(a,b)=\int%
\limits_{0}^{1}t^{a-1}(1-t)^{b-1}dt$ is a beta function.
\end{remark}

\bigskip

\begin{theorem}
\label{Theorem 2}Let $0<t_{1}<t_{2}.$ Then%
\begin{eqnarray}
P\{X_{1} &\leq &x\mid t_{1}\leq X_{r:n}\leq t_{2}\}  \notag \\
&=&\left\{
\begin{tabular}{ll}
\begin{tabular}{l}
$\left[ F(x)\sum\limits_{i=r-1}^{n-1}\binom{n-1}{i}\left[
F^{i}(t_{2})(1-F(t_{2}))^{n-1-i}\right. \right. $ \\
$\left. \left. -F^{i}(t_{1})(1-F(t_{1}))^{n-1-i}\right] \right] $ \\
$\times \left( \sum\limits_{i=r}^{n}\binom{n}{i}F^{i}(t)(1-F(t))^{n-i}%
\right) ^{-1}$%
\end{tabular}%
$,$ & $x\leq t_{1}$ \\
&  \\
\begin{tabular}{l}
$\left[ F(x)\sum\limits_{i=r-1}^{n-1}\binom{n-1}{i}%
[F^{i}(t_{2})(1-F(t_{2}))^{n-1-i}\right. -$ \\
$-F^{i}(t_{1})(1-F(t_{1}))^{n-1-i}]$ \\
$\left. -\binom{n-1}{r-1}(F(x)-F(t_{1}))F^{r-1}(t_{1})(1-F(t_{1}))^{n-r}%
\right] $ \\
$\times \left( \sum\limits_{i=r}^{n}\binom{n}{i}F^{i}(t)(1-F(t))^{n-i}%
\right) ^{-1}$%
\end{tabular}%
, & $t_{1}\leq x\leq t_{2}$ \\
&  \\
$%
\begin{tabular}{l}
$\left[ F(x)\left[ \sum\limits_{i=r-1}^{n-1}\binom{n-1}{i}%
[F^{i}(t_{2})(1-F(t_{2}))^{n-1-i}\right. \right. $ \\
$-F^{i}(t_{1})(1-F(t_{1}))^{n-1-i}]$ \\
$-\binom{n-1}{r-1}[(F(x)-F(t_{2}))F^{r-1}(t_{2})(1-F(t_{2}))^{n-r}$ \\
$\left. \left. -(F(x)-F(t_{1}))F^{r-1}(t_{1})(1-F(t_{1}))^{n-r}\right] %
\right] $ \\
$\times \left( \sum\limits_{i=r}^{n}\binom{n}{i}F^{i}(t)(1-F(t))^{n-i}%
\right) ^{-1}$%
\end{tabular}%
,$ & $x>t_{2}$%
\end{tabular}%
\right. .  \label{5a}
\end{eqnarray}%
It is clear that for $1<r<n$ (\ref{5a}) can be written as
\begin{eqnarray}
P\{X_{1} &\leq &x\mid t_{1}\leq X_{r:n}\leq t_{2}\}  \notag \\
&=&\left\{
\begin{tabular}{lll}
$%
\begin{tabular}{l}
$F(x)\left[ I_{F(t_{2})}(r-1,n-r+1)-I_{F(t_{1})}(r-1,n-r+1)\right] $ \\
$\times \left( I_{F(t_{2})}(r,n-r+1)-I_{F(t_{1})}(r,n-r+1)\right) ^{-1}$%
\end{tabular}%
$ & if & $x<t_{1}$ \\
&  &  \\
$%
\begin{tabular}{l}
$\{F(x)[I_{F(t_{2})}(r-1,n-r+1)-I_{F(t_{1})}(r-1,n-r+1)]$ \\
$+\binom{n-1}{r-1}(F(x)-F(t_{1}))F^{r-1}(t_{1})(1-F(t_{1}))^{n-r}\}$ \\
$\times \left( I_{F(t_{2})}(r,n-r+1)-I_{F(t_{1})}(r,n-r+1)\right) ^{-1}$%
\end{tabular}%
$ & if & $t_{1}\leq x\leq t_{2}$ \\
&  &  \\
$%
\begin{tabular}{l}
$\{F(x)[I_{F(t_{2})}(r-1,n-r+1)-I_{F(t_{1})}(r-1,n-r+1)]$ \\
$-\binom{n-1}{r-1}(F(x)-F(t_{2}))F^{r-1}(t_{2})(1-F(t_{2}))^{n-r}$ \\
$+\binom{n-1}{r-1}(F(x)-F(t_{1}))F^{r-1}(t_{1})(1-F(t_{1}))^{n-r}\}$ \\
$\times \left( I_{F(t_{2})}(r,n-r+1)-I_{F(t_{1})}(r,n-r+1)\right) ^{-1}$%
\end{tabular}%
$ & if & $x>t_{2}$%
\end{tabular}%
\right. .  \label{7}
\end{eqnarray}
\end{theorem}

and also as
\begin{eqnarray}
P\{X_{1} &\leq &x\mid t_{1}\leq X_{r:n}\leq t_{2}\}  \notag \\
&=&\left\{
\begin{tabular}{ll}
$%
\begin{tabular}{l}
$F(x)\left[ I_{F(t_{2})}(r-1,n-r+1)-I_{F(t_{1})}(r-1,n-r+1)\right] $ \\
$\times \left( I_{F(t_{2})}(r,n-r+1)-I_{F(t_{1})}(r,n-r+1)\right) ^{-1}$%
\end{tabular}%
,$ & $x<t_{1}$ \\
&  \\
$%
\begin{tabular}{l}
$\{F(x)I_{F(t_{2})}(r-1,n-r+1)-(F(x)-F(t_{1}))I_{F(t_{1})}(r,n-r)$ \\
$-F(t_{1})I_{F(t_{1})}(r-1,n-r+1)\}$ \\
$\times \left( I_{F(t_{2})}(r,n-r+1)-I_{F(t_{1})}(r,n-r+1)\right) ^{-1}$%
\end{tabular}%
,$ & $t_{1}\leq x\leq t_{2}$ \\
&  \\
$%
\begin{tabular}{l}
$\{(F(x)-F(t_{2}))I_{F(t_{2})}(r,n-r)+F(t_{2})I_{F(t_{2})}(r-1,n-r+1)$ \\
$-(F(x)-F(t_{1}))I_{F(t_{1})}(r,n-r)-F(t_{1})I_{F(t_{1})}(r-1,n-r+1)\}$ \\
$\times \left( I_{F(t_{2})}(r,n-r+1)-I_{F(t_{1})}(r,n-r+1)\right) ^{-1}$%
\end{tabular}%
,$ & $x>t_{2}.$%
\end{tabular}%
\right. .  \label{7a}
\end{eqnarray}

\subsection{\protect\bigskip Dependency}

Consider now the conditional random variables $Y_{1}^{(n,t)}=(X_{1}\mid
X_{n:n}\leq t)$, $Y_{2}^{(n,t)}=(X_{2}\mid X_{n:n}\leq
t),...,Y_{n}^{(n,t)}=(X_{n}\mid X_{n:n}\leq t).$

\begin{proposition}
\label{Proposition 1}$Y_{1}^{(n,t)},Y_{2}^{(n,t)},...,Y_{n}^{(n,t)}$ are iid
and \ $Y_{1}^{(n,t)}\overset{d}{=}(X_{1}\mid X_{1}\leq t)$
\end{proposition}

\begin{proof}
The joint distributions of random variables $Y_{1}^{(n,t)}$ and $%
Y_{2}^{(n,t)}$ \ can be easily found as follows:
\end{proof}

\begin{eqnarray*}
P\{Y_{1}(n,t) &\leq &x_{1},Y_{2}(n,t)\leq x_{2}\} \\
&=&\frac{P\{X_{1}\leq x_{1},X_{2}\leq x_{2},X_{n:n}\leq t\}}{P\{X_{n:n}\leq
t\}} \\
&=&\frac{P\{X_{1}\leq x_{1},X_{2}\leq x_{2},X_{1}\leq t,...,X_{n}\leq t\}}{%
F^{n}(t)}
\end{eqnarray*}%
\begin{equation*}
=\left\{
\begin{tabular}{lll}
$\frac{F(x_{1})F(x_{2})}{F^{2}(t)}$ & if & $x_{1}<x_{2}<t$ or $x_{2}<x_{1}<t$
\\
&  &  \\
$\frac{F(x_{1})}{F(t)}$ & if & $x_{1}<t<x_{2}$ \\
&  &  \\
$\frac{F(x_{2})}{F(t)}$ & if & $x_{2}<t<x_{1}$ \\
&  &  \\
$1$ & if & $x_{1}>t$ and $x_{2}>t$%
\end{tabular}%
\right.
\end{equation*}%
\begin{equation*}
=\left\{
\begin{tabular}{lll}
$\frac{F(x_{1})}{F(t)}$ & if & $x_{1}\leq t$ \\
$1$ & if & $x_{1}>t$%
\end{tabular}%
\right. \times \left\{
\begin{tabular}{lll}
$\frac{F(x_{2})}{F(t)}$ & if & $x_{2}\leq t$ \\
$1$ & if & $x_{2}>t$%
\end{tabular}%
\right.
\end{equation*}%
\begin{equation*}
=P\{Y_{1}(n,t)\leq x_{1}\}P\{Y_{2}(n,t)\leq x_{2}\}
\end{equation*}

\bigskip

Consider the random variables $Z_{1}^{(n,t)}=(X_{1}\mid X_{1:n}\leq t)$, $%
Z_{2}^{(n,t)}=(X_{2}\mid X_{1:n}\leq t),...,Z_{n}^{(n,t)}=(X_{n}\mid
X_{1:n}\leq t).$

\begin{proposition}
\label{Proposition 2} The random variables $%
Z_{1}^{(n,t)},Z_{2}^{(n,t)},...,Z_{n}^{(n,t)}$ \ are dependent.
\end{proposition}

\begin{proof}
Applying the total probability formula one can write
\begin{eqnarray*}
P\{Z_{1}(n,t) &\leq &x_{1},Z_{2}(n,t)\leq x_{2}\} \\
&=&\frac{P\{X_{1}\leq x_{1},X_{2}\leq x_{2},X_{1:n}\leq t\}}{P\{X_{1:n}\leq
t\}} \\
&=&\frac{P\{X_{1}\leq x_{1},X_{2}\leq x_{2}\}-P\{X_{1}\leq x_{1},X_{2}\leq
x_{2},X_{1:n}>t\}}{1-(1-F(t))^{n}}
\end{eqnarray*}%
\begin{equation}
=\left\{
\begin{tabular}{ll}
$(1-(1-F(t))^{-n}F(x_{1},x_{2}),$ & $x_{1}<t$ or $x_{2}<t$ \\
&  \\
\begin{tabular}{l}
$(1-(1-F(t))^{-n}[F(x_{1},x_{2})$ \\
$-(F(x_{1})-F(t))(F(x_{2})-F(t))(1-F(t))^{n-2}$%
\end{tabular}%
, & $x_{1}>t$ and $x_{2}>t$%
\end{tabular}%
\right. .  \label{8}
\end{equation}%
Comparing (\ref{8}) with (\ref{6}), it can be observed that $Z_{1}^{(n,t)}$
and $\ Z_{2}^{(n,t)}$ are dependent.
\end{proof}

\bigskip From the Proposition 2 we see that the random variables $%
Z_{1}^{(n,t)},Z_{2}^{(n,t)},...,Z_{n}^{(n,t)}$ are dependent. However, if we
consider the random variables $T_{1}^{(n,t)}=(X_{1}\mid X_{1:n}>t)$, $%
T_{2}^{(n,t)}=(X_{2}\mid X_{1:n}>t),...,T_{n}^{(n,t)}=(X_{n}\mid X_{1:n}>t),$
\ it is interesting to observe that they are iid.

\begin{proposition}
The random variables $T_{1}^{(n,t)},T_{2}^{(n,t)},...,T_{n}^{(n,t)}$ are iid
and \ $T_{1}^{(n,t)}\overset{d}{=}(X_{1}\mid X_{1}>t).$
\end{proposition}

\begin{proof}
Consider
\begin{eqnarray*}
P\{T_{1}(n,t) &\leq &x_{1},T_{2}(n,t)\leq x_{2}\} \\
&=&\frac{P\{X_{1}\leq x_{1},X_{2}\leq x_{2},X_{1:n}>t\}}{P\{X_{1:n}>t\}} \\
&=&\frac{P\{X_{1}\leq x_{1},X_{2}\leq x_{2},X_{1}>t,...,X_{n}>t\}}{%
(1-F^{n}(t))}
\end{eqnarray*}%
\begin{equation}
=\left\{
\begin{tabular}{lll}
$(1-F(t))^{-2}(F(x_{1})-F(t))(F(x_{2})-F(t))$ & if & $x_{1}>t,x_{2}>t$ \\
&  &  \\
$0$ & if & $x_{1}\leq t$ or $x_{2}\leq t$%
\end{tabular}%
\right.  \label{10}
\end{equation}%
It is clear that
\begin{eqnarray}
P\{T_{1}^{(n,t)} &\leq &x\}=P\{X_{1}\leq x\mid X_{1:n}>t\}  \notag \\
&=&\left\{
\begin{tabular}{lll}
$(1-F(t))^{-1}(F(x)-F(t))$ & if & $x>t$ \\
&  &  \\
$0$ & if & $x\leq t$%
\end{tabular}%
\right.  \label{11}
\end{eqnarray}
\end{proof}

\bigskip Comparing (\ref{10}) and (\ref{11}) we see that $%
T_{1}^{(n,t)},T_{2}^{(n,t)},...,T_{n}^{(n,t)}$ are iid random variables.

\begin{proposition}
\label{Proposition 3}Let $1\leq r<n.$ The random variables $%
Y_{1}^{(r,t)}=(X_{1}\mid X_{r:n}\leq t)$, $Y_{2}^{(r,t)}=(X_{2}\mid
X_{r:n}\leq t),...,Y_{n}^{(r,t)}=(X_{n}\mid X_{r:n}\leq t)$ are dependent.
\end{proposition}

\subsection{The absolutely continuous case and the conditional distribution
of a sample observation given an order statistic}

The conditional distribution of \ a sample in case where $F$ is absolutely
continuous underlying distribution was first considered by \ Nagaraja and
Nevzorov (1997) for a single observation and Ahmadi (2018) for multiple
observations with many interesting characterization results and applications
in reliability.

It follows from Theorem 2 that if the distribution function $F(t)$ $\ $\ is
absolutely continuous with $F^{\prime }(t)=$ $f(t),$ then \
\begin{equation*}
\lim_{h\rightarrow 0}P\{X_{1}\leq x\mid t\leq X_{r:n}\leq t+h\}
\end{equation*}%
\begin{equation*}
=\left\{
\begin{tabular}{lll}
$%
\begin{tabular}{l}
$\left( I_{F(t+h)}(r,n-r+1)-I_{F(t)}(r,n-r+1)\right) ^{-1}$ \\
$\times F(x)\left[ I_{F(t+h)}(r-1,n-r+1)-I_{F(t)}(r-1,n-r+1)\right] $ \\
\end{tabular}%
$ & if & $x<t$ \\
&  &  \\
$%
\begin{tabular}{l}
$\left( I_{F(t+h)}(r,n-r+1)-I_{F(t)}(r,n-r+1)\right) ^{-1}$ \\
$\times \{F(x)[I_{F(t+h)}(r-1,n-r+1)-I_{F(t)}(r-1,n-r+1)]$ \\
$+\binom{n-1}{r-1}(F(x)-F(t))F^{r-1}(t)(1-F(t))^{n-r}\}$%
\end{tabular}%
$ & if & $t\leq x\leq t+h$ \\
&  &  \\
$%
\begin{tabular}{l}
$\left( I_{F(t+h)}(r,n-r+1)-I_{F(t)}(r,n-r+1)\right) ^{-1}$ \\
$\times \{F(x)[I_{F(t+h)}(r-1,n-r+1)-I_{F(t)}(r-1,n-r+1)]$ \\
$-\binom{n-1}{r-1}(F(x)-F(t+h))F^{r-1}(t+h)(1-F(t+h))^{n-r}$ \\
$+\binom{n-1}{r-1}(F(x)-F(t))F^{r-1}(t)(1-F(t))^{n-r}\}$%
\end{tabular}%
$ & if & $x>t+h$%
\end{tabular}%
\right. ..
\end{equation*}%
Consider $x<t.$ \bigskip\ $\frac{1}{h}(I_{F(t+h)}(a,b)-I_{F(t)}(a,b))%
\rightarrow \frac{1}{B(a,b)}F^{a-1}(t)(1-F(t))^{b-1}f(t),$ as $h\rightarrow
0.$ Therefore, \
\begin{eqnarray*}
\lim_{h\rightarrow 0}P\{X_{1} &\leq &x\mid t\leq X_{r:n}\leq t+h\} \\
&=&\lim_{h\rightarrow 0}\frac{F(x)\left[
I_{F(t+h)}(r-1,n-r+1)-I_{F(t)}(r-1,n-r+1)\right] }{%
I_{F(t+h)}(r,n-r+1)-I_{F(t)}(r,n-r+1} \\
&=&\frac{F(x)\frac{1}{B(r-1,n-r+1)}F^{r-2}(t)(1-F(t))^{n-r}f(t)}{\frac{1}{%
B(r,n-r+1)}F^{r-1}(t)(1-F(t))^{n-r}f(t)}=\frac{r-1}{n}\frac{F(x)}{F(t)}.
\end{eqnarray*}%
Consider $t\leq x<t+h.$ Then $\frac{1}{h}(F(x)-F(t))\rightarrow f(t),$ as $%
h\rightarrow 0$ and
\begin{eqnarray*}
\lim_{h\rightarrow 0}P\{X_{1} &\leq &x\mid t\leq X_{r:n}\leq t+h\} \\
&=&\lim_{h\rightarrow 0}\left\{ \frac{%
F(x)[I_{F(t+h)}(r-1,n-r+1)-I_{F(t)}(r-1,n-r+1)]}{%
I_{F(t+h)}(r,n-r+1)-I_{F(t)}(r,n-r+1)}\right. \\
&&\left. +\frac{\binom{n-1}{r-1}(F(x)-F(t))F^{r-1}(t)(1-F(t))^{n-r}}{%
I_{F(t+h)}(r,n-r+1)-I_{F(t)}(r,n-r+1)}\right\} \\
&=&\frac{r-1}{n}\frac{F(x)}{F(t)}+\frac{1}{n}=\frac{r}{n}.
\end{eqnarray*}

\bigskip Consider $x>t+h.$ Then $\frac{\binom{n-1}{r-1}}{B(r,n-r+1)}=\frac{1%
}{n}$ and
\begin{eqnarray*}
\lim_{h\rightarrow 0}P\{X_{1} &\leq &x\mid t\leq X_{r:n}\leq t+h\} \\
&=&\lim_{h\rightarrow 0}\left\{ \frac{%
F(x)[I_{F(t+h)}(r-1,n-r+1)-I_{F(t)}(r-1,n-r+1)]}{%
I_{F(t+h)}(r,n-r+1)-I_{F(t)}(r,n-r+1)}\right. \\
&&-\frac{\binom{n-1}{r-1}}{I_{F(t+h)}(r,n-r+1)-I_{F(t)}(r,n-r+1)}%
\{(F(x)-F(t+h)) \\
&&\times F^{r-1}(t+h)(1-F(t+h))^{n-r}-(F(x)-F(t))F^{r-1}(t)(1-F(t))^{n-r}\}
\\
&=&\frac{r-1}{n}\frac{F(x)}{F(t)}-\frac{1}{n}\frac{\frac{d}{dt}\left[
(F(x)-F(t))F^{r-1}(t)(1-F(t))^{n-r}\right] }{F(t)^{r-1}(1-F(t))^{n-r}} \\
&=&\frac{(n-r)(F(x)-F(t))}{n(1-F(t))}+\frac{r}{n}
\end{eqnarray*}

\bigskip

\bigskip Therefore
\begin{eqnarray}
P\{X_{1} &\leq &x\mid X_{r:n}=t\}  \notag \\
&=&\lim_{h\rightarrow 0}P\{X_{1}\leq x\mid t\leq X_{r:n}\leq t+h\}  \notag \\
&=&\left\{
\begin{tabular}{lll}
$\frac{r-1}{n}\frac{F(x)}{F(t)}$ & if & $x<t$ \\
$\frac{(n-r)(F(x)-F(t))}{n(1-F(t))}+\frac{r}{n}$ & if & $x\geq t$%
\end{tabular}%
\right. .  \label{9}
\end{eqnarray}%
This cdf has a jump at the point $x=t$ and $F(t)-F(t-0)=\frac{1}{n}$.
Formula (\ref{8}) was first \ presented in Nagaraja and Nevzorov (1997). For
more results on conditional distributions of two or multiple observations\
given \ $X_{r:n}=t$ see the recent paper of Ahmadi and Nagaraja (2018). \ In
this comprehensive paper the joint pdf of $X_{1},X_{2},...,X_{k}$ given $%
X_{r:n}$ has been derived under the condition that $X_{1},X_{2},...,X_{k}$
are distinct than $X_{r:n}.$ The result is applied for calculating the
number of inspections that one needs to detect all failed components after
the system failure.

Below in Figure 2 we provide the graphs of cdfs and pdfs of $P\{X_{1}\leq
x\mid X_{r:n}\leq t\}$ and $P\{X_{1}\leq x\mid X_{r:n}=t\}$ in the case of
exponential underlying distributions.%
\\\\\\
\includegraphics[scale=0.15]{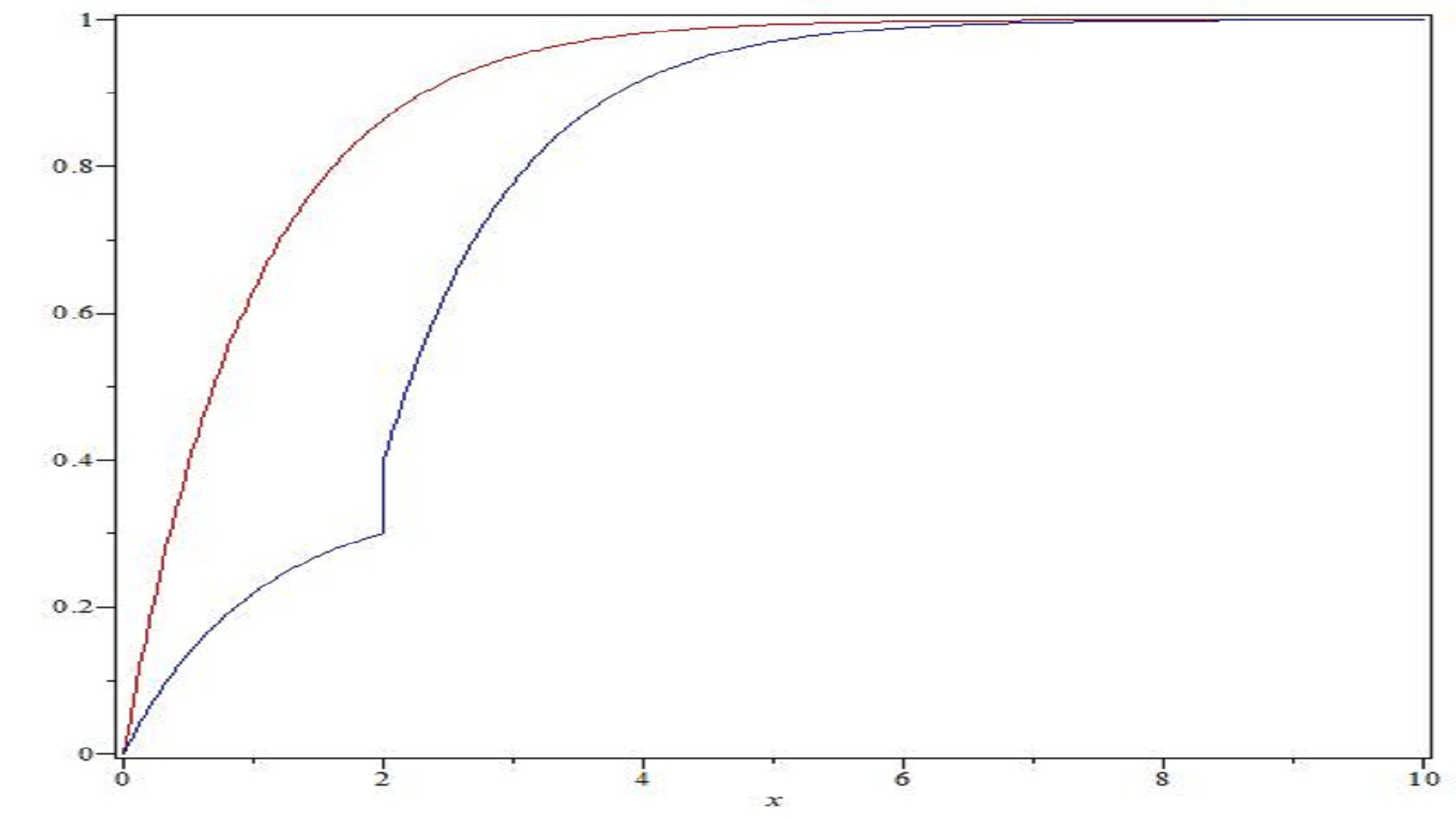}
\includegraphics[scale=0.15]{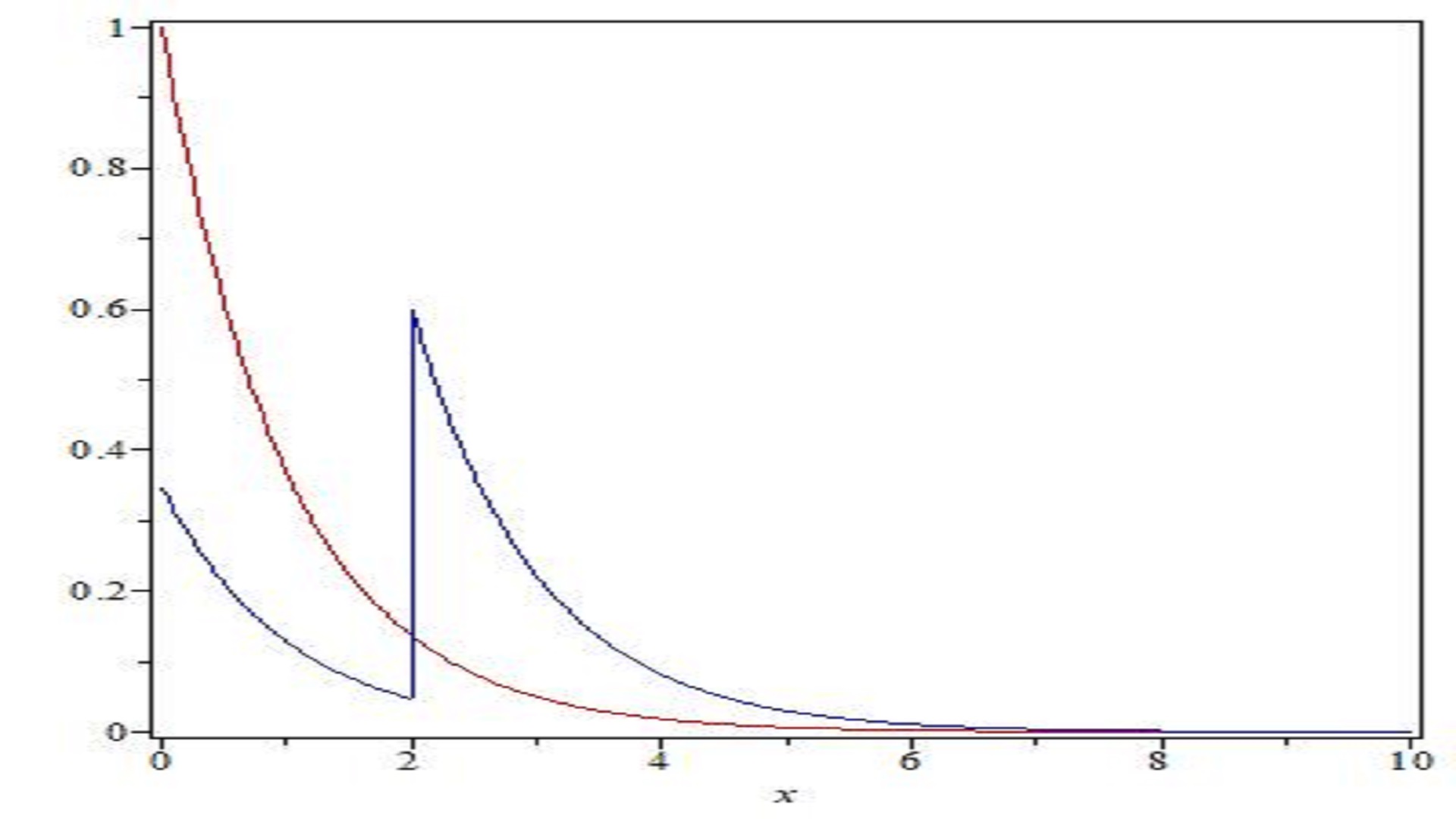}

\begin{tabular}{l}
$\text{Figure 2. The graphs of cdfs and pdfs of }P\{X_{1}\leq x\mid
X_{r:n}\leq t\}\text{ (red,above) }$ \\
$\text{and }P\{X_{1}\leq x\mid X_{r:n}=t\}\text{ (blue,below) for }%
F(x)=1-\exp (-x),x\geq 0,$ \\
$n=10,r=4,$ and $t=2$%
\end{tabular}%

\begin{remark}
\label{Remark 3}It is interesting to point out that the cdf (\ref{9}) has a
jump at the point $x=t$ and $F(t)-F(t-0)=\frac{1}{n}$, while (\ref{4}) is
continuous at the point $x=t.$
\end{remark}

\bigskip

\section{The joint distributions of a set of observations and an order
statistic}

Now we are interested in joint distributions of $X_{1},X_{2},...,X_{k}$ and $%
X_{r:n},$ $1\leq r\leq n$ and $1<k<r.$ \ Consider first $k=2.$ If $x_{1}\leq
t$ and $x_{2}\leq t,$ \ then we have
\begin{eqnarray}
F_{1,2:r}(x_{1},x_{2},t) &\equiv &  \notag \\
P\{X_{1} &\leq &x_{1},X_{2}\leq x_{2},X_{r:n}\leq t\} \\
&=&P\{X_{1}\leq x_{1},X_{2}\leq x_{2},\text{at least}  \notag \\
&&r-2\text{ of }X_{3},X_{4},...,X_{n}\text{ are less than or equal to }t\}
\notag \\
&=&F(x_{1})F_{2}(x_{2})\sum\limits_{i=r-2}^{n-2}\binom{n-2}{i}%
F^{i}(t)(1-F(t))^{n-i}  \notag \\
&=&F(x_{1})F_{2}(x_{2})I_{F(t)}(r-2,n-r+1).  \label{c1}
\end{eqnarray}%
It is clear that $F_{1,2:r}(x_{1},x_{2},t)$ has the following form
\begin{equation*}
F_{1,2:r}(x_{1},x_{2},t)=\left\{
\begin{tabular}{lll}
$F(x_{1})F(x_{2})I_{F(t)}(r-2,n-r+1),$ & $x_{1}\leq t,x_{2}\leq t$ &  \\
$\psi _{1}(x_{1},x_{2},t)$ & $x_{1}\leq t,x_{2}>t$ &  \\
$\psi _{2}(x_{1},x_{2},t)$ & $x_{1}>t,x_{2}\leq t$ &  \\
$\psi _{3}(x_{1},x_{2},t)$ & $x_{1}>t,x_{2}>t$ &
\end{tabular}%
\right.
\end{equation*}%
For the cases $x_{1}\leq t,x_{2}>t$ or $x_{1}>t,x_{2}\leq t$ o \ r $%
x_{1}>t,x_{2}>t$ the functions $\psi _{i}(x_{1},x_{2},t),$ $i=1,2,3$ can
easily be calculated by using the total probability formula considering
different cases, i.e.

\begin{eqnarray*}
P\{X_{1} &\leq &x_{1},X_{2}\leq x_{2},X_{r:n}\leq t\} \\
&=&P\{X_{1}\leq x_{1},X_{2}\leq x_{2},X_{1}\leq t,X_{2}\leq t,X_{r:n}\leq t\}
\\
+P\{X_{1} &\leq &x_{1},X_{2}\leq x_{2},X_{1}\leq t,X_{2}>t,X_{r:n}\leq t\} \\
+P\{X_{1} &\leq &x_{1},X_{2}\leq x_{2},X_{1}>t,X_{2}\leq t,X_{r:n}\leq t\} \\
+P\{X_{1} &\leq &x_{1},X_{2}\leq x_{2},X_{1}>t,X_{2}>t,X_{r:n}\leq t\}.
\end{eqnarray*}%
It is also clear that $F_{1,2:r}(x_{1},x_{2},t)$ is continuous for all
values of $x_{1},x_{2}$ and $t.$ The pdf is
\begin{equation*}
f_{1,2:r}(x_{1},x_{2},t)=\left\{
\begin{tabular}{lll}
$f(x_{1})f(x_{2})\frac{1}{B(r-2,n-r+1)}F(t)^{r-3}(1-F(t))^{n-r}f(t),$ & $%
x_{1}\leq t,x_{2}\leq t$ &  \\
$\frac{\partial ^{2}\psi _{1}(x_{1},x_{2},t)}{\partial x_{1}\partial t}$ & $%
x_{1}\leq t,x_{2}>t$ &  \\
$\frac{\partial ^{2}\psi _{2}(x_{1},x_{2},t)}{\partial x_{1}\partial t}$ & $%
x_{1}>t,x_{2}\leq t$ &  \\
$\frac{\partial ^{2}\psi _{3}(x_{1},x_{2},t)}{\partial x_{1}\partial t}$ & $%
x_{1}>t,x_{2}>t$ &
\end{tabular}%
\right.
\end{equation*}

\bigskip Actually, our interest is concentrated on the part of the joint
distribution $P\{X_{1}\leq x_{1},X_{2}\leq x_{2},...,X_{k}\leq
x_{k},X_{r:n}\leq t\}$ for $x_{1}<t,...,x_{k}<t.$

\begin{lemma}
\label{Lemma 1}For any $x_{1},x_{2},...,x_{k}\leq t$ it is true that%
\begin{eqnarray}
&&F_{1,2,...,k:r}(x_{1},x_{2},...,x_{k},t)  \notag \\
&=&\left\{
\begin{array}{ccc}
F(x_{1})F(x_{2})\cdots F(x_{k})I_{F(t)}(r-k,n-r+1), & 1\leq k<r &  \\
F(x_{1})F(x_{2})\cdots F(x_{k}) & k\geq r &
\end{array}%
\right. .  \label{ca1}
\end{eqnarray}
\end{lemma}

\begin{proof}
Indeed, if $k<r$ then $P\{X_{1}\leq x_{1},X_{2}\leq x_{2},...,X_{k}\leq
x_{k},X_{r:n}\leq t\}=P\{X_{1}\leq x_{1},X_{2}\leq x_{2},...,X_{k}\leq
x_{k}, $ at least $r-k$ of $\ X_{k+1},...X_{n}$ are less than or equal to $%
t\}=F(x_{1})F(x_{2})\cdots F(x_{k})\sum\limits_{i=r-k}^{n-k}P\{$exactly $i$
of $X_{k+1},...X_{n}$ are less than or equal to $t\}$=$\sum%
\limits_{i=r-k}^{n-k}\binom{n-k}{i}F^{i}(t)(1-F(t))^{n-k-i}.$

If $r>k,$ then it is clear that $P\{X_{1}\leq x_{1},X_{2}\leq
x_{2},...,X_{k}\leq x_{k},X_{r:n}\leq t\}=P\{X_{1}\leq x_{1},X_{2}\leq
x_{2},...,X_{k}\leq x_{k}\},$ since $x_{1},x_{2},...,x_{k}\leq t.$
\end{proof}

From the Lemma 1 it follows that \ for $x_{1},x_{2},...,x_{k}\leq t$ the
joint pdf of $X_{1},X_{2},...,X_{k},X_{r:n}$ is%
\begin{eqnarray}
&&f_{1,2,...,k:r}(x_{1},x_{2},...,x_{k},t)  \notag \\
&=&\left\{
\begin{array}{cc}
\frac{1}{B(r-k,n-r+1)}F^{r-k-1}(t)(1-F(t))^{n-r}f(t)\prod%
\limits_{i=1}^{k}f(x_{i}), & 1\leq k<r \\
f(x_{1})f(x_{2})\cdots f(x_{k}), & k\geq r.%
\end{array}%
\right. .  \label{c4}
\end{eqnarray}

\begin{theorem}
\label{Theorem 3}It is true that for $k<r,$
\begin{eqnarray*}
P\{X_{1} &\leq &X_{r:n},X_{2}\leq X_{r:n},...,X_{k}\leq X_{r:n}\} \\
&=&\frac{(n-k)!(r-1)!}{n!(r-k-1)!}.
\end{eqnarray*}
\end{theorem}

\begin{proof}
Let $k<r.$ Then using (\ref{c4}) one can write \
\begin{eqnarray*}
P\{X_{1} &<&X_{r:n},X_{2}<X_{r:n},...,X_{k}<X_{r:n}\} \\
&=&\idotsint\limits_{\{(x_{1},x_{2},...,x_{k},t):x_{1}\leq t,x_{2}\leq
t,...,x_{k}\leq
t\}}f_{1,2,...,k:r}(x_{1},x_{2},...,x_{k},t)dx_{1}dx_{2}\cdots dx_{k}dt \\
&=&\frac{1}{B(r-2,n-r+1)}\int\limits_{0}^{\infty }\int\limits_{0}^{t}\cdots
\int\limits_{0}^{t}F^{r-k-1}(t)(1-F(t))^{n-r}f(t) \\
&&\times \prod\limits_{i=1}^{k}f(x_{i})dx_{1}dx_{2}\cdots dx_{k}dt \\
&=&\frac{1}{B(r-2,n-r+1)}\int\limits_{0}^{1}t^{r-1}(1-t)^{n-r}dt=\frac{%
B(r,n-r+1)}{B(r-2,n-r+1)}=\frac{(n-k)!(r-1)!}{n!(r-k-1)!}.
\end{eqnarray*}%
The theorem is thus proved.
\end{proof}

\section{Number of inspections we need in order to detect failed components
in an (n-r+1)-out-of-n system}

Consider a coherent system with $(n-r+1)-$out-of$-n$ structure and assume
that the cdf \ $F(x)=P\{X_{1}\leq x\}$ is absolutely continuous and $%
F^{\prime }(x)=f(x).$ The $(n-r+1)-$out-of$-n$ $\ $system is intact until at
least $n-r+1$ of the components are alive, and it fails if the number of
failed components exceed $r,$ and the lifetime of this system is $%
T(X_{1},X_{2},...,X_{n})=X_{r:n}$. \ Assume that under periodical
inspections we get information about the state of the system and replace the
failed components with functioning ones. We are interested in the following
problem: in $(n-r+1)$-out-of-$n$ system some $k<r$ components may fail, the
system, however, will still be working (because of $(n-r+1)$-out-of-$n$ \
structure)$.$ In the planning of periodical inspections to detect failed
components and replace them with the working ones, an important question is:
what is the probability that we need $m$ inspections to detect $k$ failed
components? \ \ Define a random variable $N_{r:n}(k)$ to be a number of
periodical inspections we need to detect $k$ failed components. The expected
value of $N_{r:n}(k)$ will be a required average number. In engineering
designs of many technical systems, the cost of inspections is high, and
information about the expected number of inspections may reduce expenses. To
understand the random variable $N_{r:n}(k)$ we assume that the components of
the system are shown as $A_{1},A_{2},...,A_{n}$ and the corresponding
lifetimes are $X_{1},X_{2},...,X_{n}.$ For example, let $n=6,k=2,r=4,$ then $%
N_{4:6}(2)=2$ means that in two inspections we detect $2$ failed items. This
can be done as follows: in the first inspection we see that $A_{1}$ is
failed, (therefore, we must have $X_{1}<X_{4:6})$ and in the second
inspection we see that $A_{2}$ is failed (then, we must have $%
X_{2}<X_{4:6}). $ If $N_{4:6}(2)=3$ this means that we detect $2$ failed
components in $3$ inspections, and this can be done as follows: $\ $\ in the
first inspection, we have $A_{1}$ failed, in the second inspection we have $%
A_{2}$ is alive, and in the third inspection we have $A_{3}$ failed; or in
the first inspection we have $A_{1}$ is alive, in the second inspection we
have $A_{2}$ failed and in the third inspection we have $A_{3}$ failed. \
These events can be represented in strings of zeros and ones as follows: $\
N_{4:6}(2)=2$ $\Leftrightarrow \{11\};$ $N_{4:6}(2)=3$ $\Leftrightarrow
\{101,011\};$ $N_{4:6}(2)=4\Leftrightarrow \{1001,0101,0011\},$ $%
N_{4:6}(2)=5 $ $\Leftrightarrow \{10001,01001,00101,00011\}$ etc. The
following theorem allows to calculate the distribution of random variable $%
N_{r:n}(k).$

\begin{theorem}
\label{Theorem 4} \ For $1\leq k<r$ it is true that%
\begin{eqnarray}
P\{N_{r:n}(k) &=&m\}  \notag \\
&=&\binom{m-1}{k-1}\sum\limits_{j=0}^{m-k}(-1)^{j}\binom{m-k}{j}\frac{%
(n-k-j)!(r-1)!}{n!(r-k-j-1)!}  \notag \\
m &=&k,...,n-r+k+1  \label{c4a}
\end{eqnarray}
\end{theorem}

\begin{proof}
Consider the random variables
\begin{equation*}
\xi _{i}=\left\{
\begin{tabular}{lll}
$1$ & if & $X_{i}\leq X_{r:n}$ \\
$0$ & , & otherwise%
\end{tabular}%
\right. ,i=1,2,...,n.
\end{equation*}%
It is clear that $\xi _{1},\xi _{2},...,\xi _{n}$ are exchangeable. This can
be easily understood by considering, for example, the following two
probabilities:
\begin{eqnarray*}
P\{\xi _{1} &=&1,\xi _{2}=0\}=P\{X_{1}\leq
X_{r:n},X_{2}>X_{r:n}\}=P\{X_{1}\leq X_{r:n}\}-P\{X_{1}\leq
X_{r:n},X_{2}\leq X_{r:n}\} \\
P\{\xi _{1} &=&0,\xi _{2}=1\}=P\{X_{1}>X_{r:n},X_{2}\leq
X_{r:n}\}=P\{X_{2}\leq X_{r:n}\}-P\{X_{1}\leq X_{r:n},X_{2}\leq X_{r:n}\}.
\end{eqnarray*}%
We will use the following formula for exchangeable binary variables (see
George and Bowman (1995)):%
\begin{eqnarray}
P\{\xi _{1} &=&1,\xi _{2}=1,...,\xi _{k}=1,\xi _{k+1}=0,...,\xi _{m}=0\}
\notag \\
&=&\sum\limits_{j=0}^{m-k}(-1)^{j}\binom{m-k}{j}\lambda _{k+j},  \label{c5}
\end{eqnarray}%
where%
\begin{equation*}
\lambda _{k}=P\{\xi _{1}=1,\xi _{2}=1,...,\xi _{k}=1\}.
\end{equation*}%
Using exchangeability and (\ref{c5}) we have
\begin{eqnarray}
P\{N_{r:n}(k) &=&m\}  \notag \\
&=&\sum\limits_{i_{1},i_{2},...,i_{m}}P\{X_{i_{1}}\leq
X_{r:n},...,X_{i_{k}}\leq X_{r:n},X_{i_{k+1}}>X_{r:n},...,X_{i_{m}}>X_{r:n}\}
\notag \\
&=&\binom{m-1}{k-1}P\{\xi _{1}=1,\xi _{2}=1,...,\xi _{k}=1,\xi
_{k+1}=0,...,\xi _{m}=0\}  \notag \\
&=&\binom{m-1}{k-1}\sum\limits_{j=0}^{m-k}(-1)^{j}\binom{m-k}{j}\lambda
_{k+j}  \label{c6}
\end{eqnarray}%
\bigskip From the Theorem 3, one can write
\begin{eqnarray}
\lambda _{k+j} &=&P\{\xi _{1}=1,\xi _{2}=1,...,\xi _{k+j}=1\}  \notag \\
&=&\frac{(n-k-j)!(r-1)!}{n!(r-k-j-1)!}  \label{c7}
\end{eqnarray}%
\bigskip Taking (\ref{c7}) into account in (\ref{c6}), one obtains (\ref{c4a}%
). The theorem is thus proved.
\end{proof}

\

\textbf{Numerical example}

\begin{example}
\label{Numerical example}Below in Table 1 we present numerical values of $%
P\{N_{r:n}(k)=m\}$ for particular values of $n=12,$ $r=5,$ $k=3$ and $%
n=12,r=7,k=2,$ $m=k,k+1,...,n-r+k+1.$%
\begin{eqnarray*}
&&%
\begin{tabular}{ll}
$%
\begin{tabular}{|l|l|}
\hline
$m$ & $P\{N_{5:12}(3)=m\}$ \\ \hline
$3$ & $1/55=0.01818$ \\ \hline
$4$ & $8/165=0.04849$ \\ \hline
$5$ & $14/165=0.08485$ \\ \hline
$6$ & $4/33=0.12121$ \\ \hline
$7$ & $5/33=0.15152$ \\ \hline
$8$ & $28/165=0.16970$ \\ \hline
$9$ & $28/165=0.16970$ \\ \hline
$10$ & $8/55=0.14545$ \\ \hline
$11$ & $1/11=0.09091$ \\ \hline
\end{tabular}%
$ & $%
\begin{tabular}{|l|l|}
\hline
$m$ & $P\{N_{7:12}(2)=m\}$ \\ \hline
- & - \\ \hline
$2$ & $5/22=0.22727$ \\ \hline
$3$ & $3/11=0.27273$ \\ \hline
$4$ & $5/22=0.22727$ \\ \hline
$5$ & $5/33=0.15152$ \\ \hline
$6$ & $25/308=0.08117$ \\ \hline
$7$ & $5/154=0.03247$ \\ \hline
$8$ & $1/132=0.00756$ \\ \hline
- & - \\ \hline
\end{tabular}%
$%
\end{tabular}
\\
&&%
\begin{tabular}{l}
$\text{Table 1. Values of }P\{N_{r:n}(k)=m\}$ \\
$n=12,r=5,k=3,m=3,4,...,11$ (left) \\
$n=12,r=7,k=2,m=2,3,...,8$ (right)%
\end{tabular}%
\end{eqnarray*}
\end{example}

The expected value of \textbf{\ }$N_{5:12}(3)$ is
\begin{equation*}
EN_{5:12}(3)=\sum\limits_{m=3}^{11}mP\{N_{5:12}(3)=m\}=7.8,
\end{equation*}%
and the expected value of $N_{7:12}(2)$ is
\begin{equation*}
EN_{5:12}(3)=\sum\limits_{m=3}^{11}mP\{N_{7:12}(2)=m\}=3.7143.
\end{equation*}%
Therefore, in $8$-out-of-$12$ system for detecting three failed components,
we need an average of $8$ inspections and for a $6$-out-of-$12$ system to
detect $2$ failed components, we need an average of$\ 4$ inspections. \

\subsection{A discussion on mean residual and mean past functions}

Consider a coherent system with $(n-r+1)-$out-of$-n$ structure with life
times of the components having cdf $F$ and pdf $f.$\ Assume that under
periodical inspections, we get information about the state of the system.
For example, we may know that at inspection time $t_{1}$ system was
functioning, but at the next inspection time $t_{2}>t_{1}$ it appeared to
have failed. The exact failure time, however, is not known, i.e. it is
censured in time interval $(t_{1},t_{2}).$ \ One may be interested in
residual life of any of the components having this information, i.e. the
conditional mean residual life function of the components given that the
failure of the system has occurred at time interval $(t_{1},t_{2})$. \ More
precisely, we consider a function
\begin{equation*}
\varphi _{n}(t_{1},t_{2})=E\{X_{1}-t_{2}\mid t_{1}<X_{r:n}<t_{2}\}
\end{equation*}%
and call it the mean residual life (MRL) function of the component of system
failed in $(t_{1},t_{2}).$ \ Another important function is
\begin{equation*}
\psi _{n}(t_{1},t_{2})=E\{t_{2}-X_{2}\mid t_{1}<X_{r:n}<t_{2}\},
\end{equation*}%
the mean past (MP) function of the components given that the system has
failed in $(t_{1},t_{2}).$ These two functions may be important for
reliability engineers, because the knowledge of $\varphi _{n}(t_{1},t_{2})$
will help to determine expected residual life of the components having
information about the censured failure time of the system under periodical
inspections. The function $\psi _{n}(t_{1},t_{2})$ provide information about
the inactivity time of the components under the conditions described above.
\ The pdf of conditional distribution $P\{X_{1}\leq x\mid t_{1}\leq
X_{r:n}\leq t_{2}\}$ is
\begin{equation}
f_{x\mid t_{1}\leq X_{r:n}\leq t_{2}}(x\mid t_{1},t_{2})  \notag
\end{equation}%
\begin{equation}
=\left\{
\begin{tabular}{lll}
$%
\begin{tabular}{l}
$\left( I_{F(t_{2})}(r,n-r+1)-I_{F(t_{1})}(r,n-r+1)\right) ^{-1}$ \\
$\times f(x)\left[ I_{F(t_{2})}(r-1,n-r+1)-I_{F(t_{1})}(r-1,n-r+1)\right] $
\\
\end{tabular}%
$ & if & $x<t_{1}$ \\
&  &  \\
$%
\begin{tabular}{l}
$\left( I_{F(t_{2})}(r,n-r+1)-I_{F(t_{1})}(r,n-r+1)\right) ^{-1}$ \\
$\times \{f(x)[I_{F(t_{2})}(r-1,n-r+1)-I_{F(t_{1})}(r,n-r)]$%
\end{tabular}%
$ & if & $t_{1}\leq x\leq t_{2}$ \\
&  &  \\
$%
\begin{tabular}{l}
$\left( I_{F(t_{2})}(r,n-r+1)-I_{F(t_{1})}(r,n-r+1)\right) ^{-1}$ \\
$\times \{f(x)[I_{F(t_{2})}(r,n-r)-I_{F(t_{1})}(r,n-r)]$ \\
\end{tabular}%
$ & if & $x>t_{2}.$%
\end{tabular}%
\right. .  \label{g1}
\end{equation}%
\ We have

\begin{eqnarray}
\varphi _{n}(t_{1},t_{2}) &=&E\{X_{1}-t_{2}\mid t_{1}<X_{r:n}<t_{2}\}  \notag
\\
&=&\frac{1}{I_{F(t_{2})}(r,n-r+1)-I_{F(t_{1})}(r,n-r+1)}\int%
\limits_{0}^{t_{1}}xf(x)dx  \notag \\
&&+\frac{I_{F(t_{2})}(r-1,n-r+1)-I_{F(t_{1})}(r,n-r)}{%
I_{F(t_{2})}(r,n-r+1)-I_{F(t_{1})}(r,n-r+1)}\int%
\limits_{t_{1}}^{t_{2}}xf(x)dx  \notag \\
&&+\frac{I_{F(t_{2})}(r,n-r)-I_{F(t_{1})}(r,n-r)}{%
I_{F(t_{2})}(r,n-r+1)-I_{F(t_{1})}(r,n-r+1)}\int\limits_{t_{2}}^{\infty
}xf(x)dx-t_{2}  \label{g2}
\end{eqnarray}%
\newline
The mean past function $\psi _{n}(t_{1},t_{2})$ can be written as follows:%
\begin{eqnarray}
\psi _{n}(t_{1},t_{2}) &=&E\{t_{2}-X_{1}\mid t_{1}<X_{r:n}<t_{2}\}  \notag \\
&=&t_{2}-\frac{1}{I_{F(t_{2})}(r,n-r+1)-I_{F(t_{1})}(r,n-r+1)}%
\int\limits_{0}^{t_{1}}xf(x)dx  \notag \\
&&-\frac{I_{F(t_{2})}(r-1,n-r+1)-I_{F(t_{1})}(r,n-r)}{%
I_{F(t_{2})}(r,n-r+1)-I_{F(t_{1})}(r,n-r+1)}\int%
\limits_{t_{1}}^{t_{2}}xf(x)dx  \notag \\
&&-\frac{I_{F(t_{2})}(r,n-r)-I_{F(t_{1})}(r,n-r)}{%
I_{F(t_{2})}(r,n-r+1)-I_{F(t_{1})}(r,n-r+1)}\int\limits_{t_{2}}^{\infty
}xf(x)dx.  \label{g3}
\end{eqnarray}%
According to Remark 3, the conditional distribution $P\{X_{1}\leq x\mid
t_{1}<X_{r:n}<t_{2}\}$ is continuous and the MRL and MP functions of the
components given that system fails in $[t_{1},t_{2}]$ can be easily
calculated from (\ref{g2}) and (\ref{g3}).

\begin{conclusion}
In this paper we consider the joint distribution of elements of random
sample and the order statistic of the same sample. \ The joint distributions
expressed in terms of binomial sums and incomplete beta functions are
presented, and the dependence \ between related conditional random variables
is discussed. The \ distribution results \ are used to solve an important
problem in reliability analysis, to detect the failed components of coherent
system. In particular, we consider the $(n-r+1)$-out-of-$n$ system which can
function even though $k$ of the components $(k<r)$ $\ $have failed. The
number of inspections we need to detect certain number of failed components
is important information which can help to control costs. In $(n-r+1)$%
-out-of-$n$ system we define a random variable $N_{r:n}(k)$ which is the
number of inspections we need to detect $k$ components and find the
distribution of this random variable. The expected value of $N_{r:n}(k)$
provide important information which can be used in the planning of
periodical inspections of coherent systems. \
\end{conclusion}

\end{document}